\documentclass[11pt]{amsart}

\usepackage{graphicx}
\usepackage{mathptmx}

\usepackage{amscd}
\usepackage{amsthm}
\usepackage{amsxtra}
\usepackage{a4wide}
\usepackage{latexsym}
\usepackage{amssymb}
\usepackage{amsfonts}
\usepackage{amsmath}
\usepackage{amsrefs}
\usepackage{mathrsfs}

\usepackage{upref}
\usepackage{txfonts}

\usepackage{todonotes}

\usepackage[bookmarksnumbered, colorlinks, plainpages]{hyperref}
\hypersetup{colorlinks=true,linkcolor=red, anchorcolor=green, citecolor=cyan, urlcolor=red, filecolor=magenta, pdftoolbar=true}

\allowdisplaybreaks

\theoremstyle{plain}
\newtheorem{thm}{Theorem}[section]
\newtheorem{lem}[thm]{Lemma}
\newtheorem{cor}[thm]{Corollary}

\theoremstyle{definition}
\newtheorem{rem}[thm]{Remark}

\newtheorem{defi}[thm]{Definition}

\numberwithin{thm}{section}
\numberwithin{equation}{section}

\def\loc{\operatorname{loc}}

\def\esup{\operatornamewithlimits{ess\,sup}}
\def\einf{\operatornamewithlimits{ess\,inf}}

\def\R{\mathbb R}

\def\BMO{\operatorname{BMO}}

\def\rn{\R^n}
\def\a{\alpha}
\def\b{\beta}
\def\o{\omega}

\def\la{\lambda}

\def\I{(0,\infty)}

\def\d{\delta}

\def\ls{\lesssim}

\def\R{\mathbb R}

\def\dual{\,^{^{\mathsf{c}}}\!}

\def\Bt {{B(0,t)}}

\def\Bxt {{B(x,t)}}
\def\Br {{B(0,r)}}
\def\Bxr {{B(x,r)}}

\def\Lploc{L_p^{\rm loc}(\rn)}

\def\Llocp{L_1^{\rm loc, +}(\rn)}
\def\Lloc{L_1^{\rm loc}(\rn)}

\begin{document}

\title[]{An extension of Muchenhoupt-Wheeden theorem to generalized weighted (central) Morrey spaces}

\author[]{Rza Mustafayev and Abdulhamit Kucukaslan}

\address{Rza Mustafayev, Department of Mathematics, Faculty of Science, Karamanoglu Mehmetbey University, Karaman, 70100, Turkey}
\email{rzamustafayev@gmail.com}

\address{Abdulhamit Kucukaslan, Czech Academy of Sciences, Institute of Mathematics, Prague, Czech Republic \\
current address: School of Applied Sciences, Pamukkale University, Denizli, Turkey}
\email{aakucukaslan@gmail.com}

\thanks{The research of Abdulhamit Kucukaslan was totally supported by the grant of The Scientific and Technological Research Council of Turkey (TUBITAK), [Grant-1059B191600675-2016-I-2219].}



\maketitle

\begin{abstract}
In this paper we find the condition on function $\omega$ and weight $v$ which ensures
the equivalency of norms of the Riesz potential and the fractional	maximal function in generalized weighted Morrey spaces ${\mathcal M}_{p,\omega}({\mathbb R}^n,v)$ and generalized weighted central Morrey spaces $\dot{\mathcal M}_{p,\omega}({\mathbb R}^n,v)$, when $v$ belongs to Muckenhoupt $A_{\infty}$-class. 
\end{abstract}

\section{Introduction}

Morrey spaces ${\mathcal M}_{p, \lambda}(\rn)$ were introduced in
\cite{Morrey} and defined as follows: For $\lambda \ge 0$, $1\le p
\le \infty$, $f\in {\mathcal M}_{p, \lambda} (\rn)$ if $f\in \Lploc$ and
$$
\left\| f \right\|_{{\mathcal M}_{p,\lambda }(\rn)}= \sup_{x\in \rn, \; r>0 }
r^{-{\lambda}/{p}} \|f\|_{L_{p}(B(x,r))} <\infty
$$
holds. These spaces appeared to be quite useful in the study of
local behavior of the solutions of partial differential equations.
Later, Morrey spaces found important applications to Navier-Stokes
(\cites{Maz,Tay}) and Schr\"{o}dinger equations
(\cite{RV1,Shen1}), elliptic equations with discontinuous
coefficients (\cite{Caf,FanLuYang,FPR}) and
potential theory (\cite{Adams,Adams1}).

Morrey spaces were widely investigated during last decades,
including the study of classical operators of harmonic analysis such as
maximal, singular and potential operators and their commutators with a measurable functions, in generalizations of
these spaces. We refer to a few works in this direction (see, for instance,
\cite{KomShi,Mus1,SamkoN,PerSam,NakamuraSaw_2017,GKMS}).

We find it convenient to define the generalized weighted Morrey spaces and the generalized weighted central Morrey spaces in the form as follows (cf. \cite{KomShi,Mus1,Nakamura_2016,PersSamWall_2016}).
\begin{defi}
	Let $1 \le p < \infty$ and $\omega(x,r)$ be a positive continuous
	function on $\rn \times (0,\infty)$. Let $v$ be a  weight
	function $\rn$. We denote by ${\mathcal	M}_{p,\omega}(\rn,v)$ the generalized weighted Morrey spaces, the space of all
	functions $f\in L_{p}^{\loc}(\rn,v)$ with finite quasinorm
	$$
	\|f\|_{\mathcal{M}_{p,\omega}(\rn,v)} = \sup\limits_{x\in\rn,\,r>0}
	\omega(x,r) \, \|f\|_{L_{p}(B(x,r),v)}.
	$$
\end{defi}

Recall that ${\mathcal M}_{p,\omega}(\rn,v)$ when $v \equiv 1$ is the generalized Morrey space ${\mathcal M}_{p,\omega}(\rn)$ introduced in \cite{Miz_1990} and \cite{Nakai}. 

\begin{defi}
	Let $1 \le p < \infty$ and $\omega$ be a positive continuous
	function on $(0,\infty)$. Let $v$ be a  weight
	function $\rn$. We denote by $\dot{\mathcal M}_{p,\omega}(v)=\dot{\mathcal
		M}_{p,\omega}(\rn,v)$ the generalized weighted central Morrey space, the space of all
	functions $f\in L_{p,v}^{\loc}(\rn)$ with finite quasinorm
	$$
	\|f\|_{\dot{\mathcal{M}}_{p,\omega}(v)} = \sup\limits_{r>0}
	\omega(r)\|f\|_{L_{p,v}(B(0,r))}.
	$$
\end{defi}
The localized (central) Morrey spaces were considered in \cites{AlLakGuz_2000,GarciaHer_1994} in order to study the relationship between central $\BMO$ spaces and Morrey spaces. The generalized weighted central Morrey spaces are the special case in the scale of the weighted local Morrey-type spaces (see, for instance, \cites{MU_2015,gmu_2017} and references given there). 

$I_{\a}f$ and $M_{\a}f$ denote the Riesz potential
and the fractional maximal function of a nonnegative locally integrable function $f$ on $\rn$, respectively:
\begin{align*}
I_{\a}f (x) & =\int_{\rn} \frac{f(y)\,dy}{|x-y|^{n-\a}},~~ 0<\a<n, \\
\intertext{and}
M_{\a}f(x) & =\sup_{r>0}\frac{1}{|B(x,r)|^{1 - \alpha / n}}\int_{B(x,r)} f(y)\,dy,~~0\leq \a <n.
\end{align*}
Here $\Bxr$ denotes the open ball centered at $x$ of radius $r$ and $|B(x,r)|$ is the Lebesgue measure of $B(x,r)$. For $\a=0$, $M_0 = M$ is the Hardy-Littlewood maximal operator.

Recall that, for $0<\a<n$ there is a constant $C > 0$ such that the inequality
\begin{equation}\label{eq001}
M_{\a}f(x)\le C\, I_{\a}f (x)
\end{equation}
holds for any nonnegative locally integrable function $f$ on $\rn$ and $x\in \rn$. The opposite inequality is in general false. We recall the following theorem of B.~Muckenhoupt and R.L.~Wheeden.
\begin{thm}\cite[Theorem 1]{MuckWheeden}
	Let $0 < p < \infty$, $0 < \alpha < n$ and $v \in A_{\infty}$. Then there is a constant $C > 0$ such that the inequalities
	$$
	C^{-1} \|M_{\alpha} f \|_{L_p(\rn,v)} \le \|I_{\alpha} f \|_{L_p(\rn,v)} \le C \|M_{\alpha} f \|_{L_p(\rn,v)}
	$$
	hold for any nonnegative locally integrable function $f$ on $\rn$.
\end{thm}

In \cite{AdXi} D.R.~Adams and J.~Xiao proved the following theorem.
\begin{thm}[\cite{AdXi}, Theorem 4.2]\label{thmAdXi}
	Let $1<p<\infty$, $0<\a<n$ and $0\leq\la< n$. Then there is a constant $C > 0$ such that the inequalities
	\begin{equation*}
	C^{-1} \|M_{\a}f\|_{{\mathcal M}_{p,\la }(\rn)} \le \|I_{\a}f\|_{{\mathcal M}_{p,\la }(\rn)} \le C\|M_{\a}f\|_{{\mathcal M}_{p,\la }(\rn)}
	\end{equation*}
	hold for any nonnegative locally integrable function $f$ on $\rn$.
\end{thm}

The following theorem was proved in \cite{gm}.
\begin{thm}[\cite{gm}, Theorem 1.4]\label{thmGM}
	Let $1<p<\infty$, $0<\a<n$ and $\o$ be a continuous weight function defined
	on $(0,\infty)$. If 
	\begin{equation*}
	\sup_{r > 0} r^{n - \a}\left(\sup_{r<s<\infty}s^{\a-n}\sup_{0<\tau<s}\o(\tau)\tau^{{n} / {p}}\right)\int_r^{\infty}t^{\a-n-1} \left(\sup_{t<s<\infty}s^{\a-n}\sup_{0<\tau<s}\o(\tau)\tau^{{n} / {p}}\right)^{-1}dt < \infty,
	\end{equation*}
	then there is a constant $C > 0$ such that  the inequalities
	\begin{equation*}
	C^{-1} \|M_{\a}f\|_{{\mathcal M}_{p,\o }(\rn)} \le \|I_{\a}f\|_{{\mathcal M}_{p,\o }(\rn)} \le C	\|M_{\a}f\|_{{\mathcal M}_{p,\o }(\rn)}
	\end{equation*}
	hold for any nonnegative locally integrable function $f$ on $\rn$.
\end{thm}

The aim of this paper is to extend Theorem \ref{thmGM} to generalized weighted Morrey spaces $\mathcal{M}_{p,\omega}(\rn,v)$ (see, Theorem \ref{equiv}) and generalized weighted central Morrey space $\dot{\mathcal M}_{p,\o}(\rn,v)$ (see, Theorem \ref{equiv_2}), when the weight function $v$ belongs to Muckenhoupts $A_{\infty}$-class. 

In order to solve this task, firstly, we study relation between weighted Lebesgue norms of $I_{\alpha}$ and $M_{\alpha}$ over cubes in $\rn$: we show that if $1<p<\infty$, $0<\a<n$ and $v\in A_{\infty}$, then the two side estimate 
\begin{equation}\label{eq.equiv}
\|I_{\a}f\|_{L_{p}(Q,v)} \thickapprox \|M_{\a}f\|_{L_{p}(Q,v)} + v(Q)^{{1} / {p}}\int_{\rn\backslash Q}\frac{f(y)dy}{|y-x_0|^{n-\a}},
\end{equation}
hold for all $f\in \Llocp$ and for any cube $Q=Q(x_0,r_0)$ with constants independent of $Q$ and $f$. Recall that this estimate is a generalization to the weighted case of \cite[Theorem 1.10]{gm}.

As we shall see, afterwards, to achieve our main goal, we need a solution for the two-operator weighted norm inequality
\begin{equation}\label{eq3248579857610}
\sup_{ r>0 }u(r)\int_{\rn\backslash	\Br}\frac{g(y)}{|y|^{\b}}dy \lesssim \sup_{r>0}u(r) \left(\sup_{t>r}t^{-\b}\int_{\Bt} g(y)dy\right).
\end{equation}
The study of inequality \eqref{eq3248579857610} has independent
interest (For the history of such type inequalities in 1-dimensional and n-dimensional cases we refer to \cite{gmu_CMJ} and \cite{gmu_2017}, respectively). 
Inequality \eqref{eq3248579857610} is a special case (when $p_1 = p_2 = 1$, $q_1 = q_2 = \infty$, $v_1(y) \equiv 1 $, $v_2 (y) = |y|^{-\beta}$, $u_1 (t) = u_2(t) = u(t)$) of the inequality
\begin{align}\label{eq.EMJ_2017}
\bigg\| \big\|f\big\|_{L_{p_2}({\,^{^{\bf c}}\!}B(0,\cdot),v_2)} \bigg\|_{L_{q_2}((0,\infty),u_2)} \le c \,\bigg\| \big\|f\big\|_{L_{p_1}(B(0,\cdot),v_1)} \bigg\|_{L_{q_1}((0,\infty),u_1)}. 
\end{align}
Inequality \eqref{eq.EMJ_2017} was studied in \cite{gmu_2017}, when $p_1,\,p_2,\,q_1,\,q_2 \in (0,\infty)$, $p_2 \le q_2$ and $u_1,\,u_2$ and $v_1,\,v_2$ are weights on $(0,\infty)$ and ${\mathbb R}^n$, respectively. In this paper we give solution of the missing case, when $q_1 = q_2 = \infty$.

Using inequalities \eqref{eq.equiv}, the characterization of \eqref{eq3248579857610} allows us to formulate a sufficient condition ensuring the equivalency of norms of $I_{\alpha}$ and $M_{\alpha}$ in generalized weighted Morrey spaces ${\mathcal M}_{p,\omega}({\mathbb R}^n,v)$, when $1<p<\infty$, $0<\a<n$ and $v \in A_{\infty}$. 	If, moreover, $v\in RD_{q(1-\a/n)}$, with some $p < q$, then we present a necessary and sufficient condition for the equivalency of norms of $I_{\alpha}$ and $M_{\alpha}$ in generalized weighted central Morrey spaces $\dot{\mathcal M}_{p,\o}(\rn,v)$.

The paper is organized as follows. We start with some notations and preliminaries in Section~\ref{pre}. In Section \ref{FrMax}, we present some weighted $L_p$-estimates for fractional maximal functions over balls.  Relation between $\|I_{\a}f\|_{L_{p}(Q,v)}$ and $\|M_{\a}f\|_{L_{p}(Q,v)}$ was investigated in Section \ref{Sect3}. In Section \ref{Sect5}, we give necessary and sufficient
condition for \eqref{eq3248579857610} to hold. Finally, in Section \ref{Sect6}, we present the condition on $\o$ and $v$ which ensures equivalence of norms of Riesz potential and fractional maximal function in generalized weighted Morrey spaces $\mathcal M_{p,\o}(\rn,v)$  and generalized weighted central Morrey spaces $\dot{\mathcal M}_{p,\o}(\rn,v)$.

%%%%%%%%%%%%%%%%%%%%%%%%%%%%%%%%%%%%%%%%%%%%%%%%%%%%%%%%%%%%%%%%%%%%%%%%%%%%%%%%%%%%%%%%%%%%%%%

\

\section{Notations and Preliminaries}\label{pre}

\

We make some conventions. Throughout the paper, we always denote by
$c$ and $C$ a positive constant which is independent of main
parameters, but it may vary from line to line. By $A\ls B$ we mean
that $A\le CB$ with some positive constant $C$ independent of
appropriate quantities. If $A\ls B$ and $B\ls A$, we write $A\approx
B$ and say that $A$ and $B$ are  equivalent. Constant, with
subscript such as $c_1$, does not change in different occurrences.
For a measurable set $E$, $\chi_E$ denotes the characteristic
function of $E$. Recall that $f_E$ denotes the mean value
$f_E=(1/|E|)\int_Ef(y)dy$ of an integrable function $f$ over a set
$E$ of positive finite measure. Given $\la>0$ and a cube $Q$, $\la
Q$ denotes the cube with the same center as $Q$ and whose side is
$\la$ times that of $Q$.

Let $\Omega$ be any measurable subset of $\rn$, $n\geq 1$. Let
${\mathfrak M} (\Omega)$ denote the set of all measurable functions on $\Omega$
and ${\mathfrak M}_0 (\Omega)$ the class of functions in ${\mathfrak M} (\Omega)$ that
are finite a.e. The symbol ${\mathfrak M}^+ (\Omega)$ stands for the
collection of all $f \in {\mathfrak M} (\Omega)$ which are non-negative on
$\Omega$. 

A weight is a locally integrable function on $\rn$ which takes
values in $(0,\infty)$ almost everywhere. With any weight function
$v$ we associate the measure  $v(E)=\int_{E}v(x)dx$. Given a weight $v$, we say that $v$
satisfies the doubling condition if there exists a constant $D>0$
such that for any cube $Q$, we have $v(2Q)\leq D v(Q)$. When $v$
satisfies this condition, we write $v\in \mathcal D$, for short.

For $p\in (0,\infty]$ and $v\in {\mathfrak M}^+(\Omega)$, we define the functional
$\|\cdot\|_{L_{p}(\Omega,v)}$ on ${\mathfrak M} (\Omega)$ by
\begin{equation*}
\|f\|_{L_{p}(\Omega,v)} : = \left\{\begin{array}{cl}
\left(\int_{\Omega} |f(x)|^p v(x)\,dx \right)^{1/p} & \qquad \mbox{if} \qquad p<\infty, \\
\esup_{x \in \Omega} |f(x)|v(x) & \qquad \mbox{if} \qquad p=\infty.
\end{array}
\right.
\end{equation*}

If, in addition, $v$ is a weight function on $\Omega$, then the weighted Lebesgue space
$L_{p}(\Omega,v)$ is given by
\begin{equation*}
L_{p}(\Omega,v) = \{f\in {\mathfrak M} (\Omega):\,\, \|f\|_{L_{p}(\Omega,v)} <
\infty\}
\end{equation*}
and it is equipped with the quasi-norm $\|\cdot\|_{L_{p}(\Omega,v)}$.

When $v\equiv 1$ on $\Omega$, we write simply $L_p(\Omega)$ and
$\|\cdot\|_{L_p(\Omega)}$ instead of $L_{p}(\Omega,v)$ and
$\|\cdot\|_{L_{p}(\Omega,v)}$, respectively.

By $L_{p}^{\loc}(\rn,v)$ we denote the set of all $f \in {\mathfrak M} (\rn)$ such that $f \in L_{p}(K,v)$ for each compact subset $K$ of $\rn$.
Let us denote by  $L_{p}^{\loc,+}(\rn,v)$ the set of all non-negative
functions from $L_{p}^{\loc}(\rn,v)$.

We say that a weight $v$ satisfies Muckenhoupt's $A_p$-condition \cite{Muck} if there
exists a constant $C > 0$ such that, for any cube $Q$,
$$
\left(\int_Q v(x)dx\right)\left(\int_Q
v(x)^{1-p'}dx\right)^{p-1}\leq C|Q|^p,
$$
where $1/p+1/p'=1$. 

We say that a weight function $v$ satisfies Muckenhoupt's $A_1$-condition \cite{Muck} if there exists $C > 0$ such that
$$
\frac{1}{|Q|} \int_Q v(x)\,dx \le C \, \einf_{y \in Q} v(y)
$$
for all cubes $Q \subset \rn$. It is easy to see \cite[p. 389]{GR} that $v \in A_1$ is equivalent to the requirement that
$$
Mv(x) \le C \, v(x) \quad a.a. \,\, x \in \rn.
$$

The class $A_{\infty}$ is the union $\bigcup_{1 \le p < \infty} A_p$.
Equivalently, $v\in A_{\infty}$ if and only if there exists two
constants $0<\d\leq 1$ and $C>0$ such that for every cube $Q$ and
every measurable set $E\subset Q$
$$
\frac{v(E)}{v(Q)}\leq C \left( \frac{|E|}{|Q|}\right)^{\d}
$$
holds. 

As references for the $A_p$ classes we give \cite{Muck,GR,graf}.

We say that a weight $v$ satisfies the reverse doubling condition of
order $\b$ with $0<\b<\infty$ if
$$
\frac{v(B')}{v(B)} \leq c
\left(\frac{|B'|}{|B|}\right)^{\b} \qquad \mbox{for all balls} \quad
B'\subset B,
$$
and write $v\in RD_{\b}$ in this case (see, \cites{Perez,SawWheed_1992}). 

A weıght $v$ satisfies $\beta$-dimensional $A_{\infty}$ condition  ($v \in A_{\infty}^{\beta}$) if there are positive constants $C,\,\delta$ such that
$$
\frac{v(E)}{v(Q)} \le C\, \left( \frac{\|E\|_{\beta,Q}}{|Q|^{\beta}}\right)^{\delta},
$$
whenever $E$ is a measurable subset of a cube $Q$ in $\rn$. Here
$$
\|E\|_{\beta,Q} : = \inf \left\{ \sum_{i} |Q_i|^{\beta} :\, E \subset \bigcup_i Q_i \subset Q \right\}. 
$$
Recall that $RD_{\b} \subset A_{\infty}^{\beta}$ and $A_{\infty}^1 = A_{\infty}$ (see, \cite[p. 818]{SawWheed_1992}).

We recall the definitions of standart harmonic analysis tools, such as the sharp maximal function and so-called local sharp maximal function. 

The Fefferman-Stein \cite{fefstein_1972} maximal function $f^{\#}$ and the John-Str\"omberg \cites{john_1965,strom_1979} maximal function $M_{\lambda}^{\#}f$ defined for a measurable function $f$ and $x \in \rn$ by
\begin{align*}
f^{\#}(x) : & = \sup_{x \in Q} \, \inf_{c \in \mathbb C} \, \frac{1}{|Q|}\int_{Q} |f(y)-c|\,dy, \\
\intertext{and} 
M_{\lambda}^{\#}f (x) : & = \sup_{x \in Q} \, \inf_{c \in \mathbb C} \, ((f-c)\chi_{Q})^* (\lambda |Q|), \quad 0 < \lambda \le 1,
\end{align*}
respectively, where the supremum is taken over all cubes $Q$ containing $x$. 

Recall that the space $\BMO (\rn)$ \cite{johnNir_1961} is directly generated by $f^{\#}$:
$$
\|f\|_{\BMO (\rn)} = \|f^{\#}\|_{L_{\infty}(\rn)},
$$
while $M_{\lambda}^{\#}f$ gives an alternative characterization of $\BMO (\rn)$:
\begin{equation}\label{eq.BMO_2}
\lambda \|M_{\lambda}^{\#}f \|_{L_{\infty}(\rn)} \le \|f\|_{\BMO (\rn)} \le c_n \|M_{\lambda}^{\#}f \|_{L_{\infty}(\rn)},\quad 0 < \lambda \le 1/2.
\end{equation}
The first estimate in \eqref{eq.BMO_2} holds by Chebyshev's inequality, while the second one is a deep result due to John \cite{john_1965} and Str\"omberg \cite{strom_1979}.

A close relation between $M_{\lambda}^{\#}f$ and $f^{\#}$ is provided by the following statement.
\begin{lem}\cite[Lemma 3.4]{jawtor_1985}\label{jawtor_1985}
	There exists $0 < \lambda_n < 1$ and $c,\,C > 0$ such that the inequality
	$$
	c f^{\#} (x) \le M M_{\lambda}^{\#} f(x) \le C f^{\#}(x)
	$$
	holds for all $f \in L_1^{\loc}(\rn)$ and $x \in \rn$ provided that  $0 < \lambda \le \lambda_n$.
\end{lem}

The following statement follows by \cite[Theorem 8]{NakamuraSaw_2017}.
\begin{lem}\label{NakSaw_2017}
	Let $0 < p < \infty$ and $v \in A_{\infty}$. There exist $0 < \lambda_v < 1$ and  $c_{n,p,v} > 0$ such that the inequality  
	$$
	\|f\|_{L_p(Q,v)} \le c_{n,p,v} \, \bigg( v(Q)^{1/p}|f|_Q + \big\|M_{\lambda_v}^{\#} f \big\|_{L_{p}(Q,v)} \bigg)
	$$
	holds for any $f \in {\mathfrak M}(\rn)$ and  $Q \subset \rn$.
\end{lem}

Combining Lemmas \ref{jawtor_1985} and \ref{NakSaw_2017}, we have the following lemma.
\begin{lem}\label{lem876346754}
	Let $0 < p < \infty$ and $v \in A_{\infty}$. There exist $0 < \lambda_v < 1$ and  $c_{n,p,v} > 0$ such that the inequality
	\begin{equation*}
	\|f\|_{L_{p}(Q,v)} \leq c \left (v(Q)^{{1} / {p}}|f|_Q +
	\|f^{\#}\|_{L_{p}(Q,v)}\right)
	\end{equation*}
	holds for any $f \in L_{1}^{\loc}(\rn)$ and each cube $Q \subset \rn$.	
\end{lem}

\begin{rem}
Lemma \ref{lem876346754}, in particular, states that there exists $c_{n,p} > 0$ such that the inequality  
$$
\|f\|_{L_{p}(Q)} \leq c_{n,p} \left (|Q|^{{1} / {p}}|f|_Q + \|f^{\#}\|_{L_{p}(Q)}\right)
$$
holds for any $f \in L_{1}^{\loc}(\rn)$, and each cube $Q \subset \rn$.	This note shows that in inequality (7.13) of \cite[Corollary 7.5, p. 380]{BS} the multiplier $|Q|^{1/p}$ in front of $|f|_Q$ had been lost.
\end{rem}

\

\section{Weighted $L_p$-estimates for fractional maximal functions over balls}\label{FrMax}

\

In this section we obtain weighted $L_p$-estimates for fractional maximal functions over balls. 

The non-increasing rearrangement (see, e.g., \cite[p. 39]{BS}) of a function
$f \in {\mathfrak M}_0 (\rn)$ with respect to the measure $v(x)\,dx$ is
defined by
$$
f_v^*(t)=\inf\left\{\la >0 : v(\{x\in\rn: |f(x)|>\la \}) \leq t
\right\}\quad (0<t<\infty).
$$
Let $p\in[1,\infty)$. The \textit{weighted weak Lorentz space}
$L^{p,\infty}(\rn,v)$ is defined as
$$
L^{p,\infty}(\rn,v):= \left\{f \in {\mathfrak M}_0(\rn): \quad
\|f\|_{L^{p,\infty}(\rn,v)}:=\sup_{0<t<\infty}t\sp{1 / {p}}f_v^*(t)<\infty\right\}.
$$

In order to get an estimate from above we need the following lemma.
\begin{lem}\label{lem geydime}
	Let $0\leq \a <n$, $1 \leq q<\infty$, $0<p<q<\infty$,   $v$ be a weight
	function on $\rn$. Then the inequality
	\begin{equation}\label{equrty82tyqwert}
	\|M_{\a}(f\chi_{(2B)})\|_{L_{p}(B,v)} \lesssim
	v(B)^{1/p-1/q}\left(\sup_{B'\subset
		B}\frac{v(B')}{|B'|^{q(1-\a/n)}}\right)^{1/q}\|f\|_{L_1(2B)}  ,
	\end{equation}
	holds for any ball $B$ in $\rn$ and $f\in \Lloc$ with constant which does not depend on $B$ and $f$.
\end{lem}

\begin{proof}
	Let $B$ be any ball in $\rn$. Applying the Hardy-Littlewood inequality, we get that
	\begin{align}\label{eq3189657982176}
	\left(\int_{B}\left|f(x)\right|^{p}v(x)dx\right)^{1/p} &\leq \left(\int_0^{v(B)}\left[(\chi_B f)^*_v(t)\right]^{p}dt \right)^{1/p} \notag \\
	&\leq \left[\sup_{0<t<v(B)}t^{1/q}(\chi_B f)^*_v(t)\right]
	\left(\int_0^{v(B)}t^{-p/q}dt\right)^{1/p} \notag \\
	&\thickapprox\|\chi_Bf\|_{L^{q,\infty}(\rn,v)}v(B)^{1/p-1/q}.
	\end{align}
	Therefore
	\begin{equation}\label{eq25076108792346t1}
	\|\chi_B M_{\a}(f\chi_{(2B)})\|_{L_{p}(\rn,v)}\lesssim
	v(B)^{1/p-1/q} \|\chi_BM_{\a}(f\chi_{(2B)})\|_{L^{q,\infty}(\rn,v)}.
	\end{equation}
	Put
	$$
	E_s=\{x\in B: M_{\a}(f\chi_{(2B)})(x)>s\}.
	$$
	For every point $x\in E_s$ we can find a ball $B_x=B(x,r(x))\subset
	3B$ such that
	$$
	|B_x|^{\a/n-1}\int_{B_x}|f(y)|\chi_{2B}(y)dy>s.
	$$
	The family of balls $\{B_x\}_{x\in E_s}$ covers the bounded set
	$E_s$. By Vitali covering lemma, this family contains a
	sequence of non-intersecting balls $\{B_k\}_k$.	Thus
	\begin{align*}
	v(E_s)\leq \sum\limits_{k=1}^{\infty} v(3B_k)&\leq
	\sum\limits_{k=1}^{\infty} v(3B_k) s^{-q}
	\left(|B_k|^{\a/n-1}\int_{B_k}|f(y)|\chi_{2B}(y)dy\right)^{q}\\
	& \leq C s^{-q}\sum\limits_{k=1}^{\infty}
	\frac{v(3B_k)}{|3B_k|^{q(1-\a/n)}}
	\left(\int_{B_k}|f(y)|\chi_{2B}(y)dy\right)^{q}\\
	&\leq C s^{-q}\sup\limits_{B'\subset
		B}\frac{v(B')}{|B'|^{q(1-\a/n)}}\sum\limits_{k=1}^{\infty}
	\left(\int_{B_k}|f(y)|\chi_{2B}(y)dy\right)^{q}\\
	&\leq C s^{-q}\sup\limits_{B'\subset
		B}\frac{v(B')}{|B'|^{q(1-\a/n)}}\left(\sum\limits_{k=1}^{\infty}
	\int_{B_k}|f(y)|\chi_{2B}(y)dy\right)^{q}\\
	&\leq C s^{-q}\sup\limits_{B'\subset
		B}\frac{v(B')}{|B'|^{q(1-\a/n)}}\|f\|_{L_1(2B)}^q.
	\end{align*}
	Therefore
	\begin{align}\label{eq12539867289754}
	\|\chi_BM_{\a}(f\chi_{(2B)})\|_{L^{q,\infty}(\rn,v)} = \sup_{0<s<\infty}s \, v(E_s)^{1/q} \leq C \left(\sup\limits_{B'\subset
		B}\frac{v(B')}{|B'|^{q(1-\a/n)}}\right)^{1/q}\|f\|_{L_1(2B)}.
	\end{align}
	In view of \eqref{eq3189657982176}, \eqref{eq25076108792346t1} and
	\eqref{eq12539867289754}, we arrive at \eqref{equrty82tyqwert}.
\end{proof}

If we restrict weight function to be from  reverse doubling class of some order, we get the following statement.
\begin{cor}
	Let  $0\leq \a <n$, $1 \leq q<\infty$, $0<p<q<\infty$,  $v\in
	RD_{q(1-\a/n)}$. Then the inequality 
	\begin{equation*}
	\|M_{\a}(f\chi_{(2B)})\|_{L_{p}(B,v)} \lesssim
	v(B)^{1/p}|B|^{\a/n-1}\|f\|_{L_1(2B)},
	\end{equation*}
	holds for any ball $B$ in $\rn$ and $f\in \Lloc$ with constant independent of $B$ and $f$.
\end{cor}
\begin{proof}
	Since $v\in RD_{q(1-\a/n)}$, we have
	$$
	\left(\sup_{B'\subset B}\frac{v(B')}{|B'|^{q(1-\a/n)}}\right)^{1/q}
	\leq \frac{v(B)^{1/q}}{|B|^{1 - \a / n}}.
	$$
	Thus, by Lemma \ref{lem geydime}, we get
	\begin{equation*}
	\|M_{\a}(f\chi_{(2B)})\|_{L_{p}(B,v)} \lesssim
	v(B)^{1/p-1/q}v(B)^{1/q}|B|^{\a/n-1}\|f\|_{L_1(2B)}\thickapprox
	v(B)^{1/p}|B|^{\a/n-1}\|f\|_{L_1(2B)}.
	\end{equation*}
\end{proof}

The following lemma is true.
\begin{lem}\label{lem.RD}
	Let $1 \leq q<\infty$, $0<p<q<\infty$, $0\leq \a <n$, $v\in RD_{q(1-\a/n)}$
	and $f\in \Lloc$. Then for any ball $B=B(x,r)\subset\rn$
	\begin{equation}\label{25089672897}
	\|M_{\a}f\|_{L_{p}(B,v)} \lesssim
	v(B)^{{1} / {p}}\left(\sup_{t>r}|B(x,t)|^{\a/n-1}\int_{B(x,t)}|f(y)|dy\right),
	\end{equation}
	where constants in equivalency do not depend on $B$ and $f$.
\end{lem}
\begin{proof}
	It is obvious that for any ball $B=B(x,r)$
	\begin{equation}\label{eq678753456}
	\|M_{\a} f\|_{L_{p}(B,v)} \leq \|M_{\a}
	(f\chi_{(2B)})\|_{L_{p}(B,v)}+ \|M_{\a} (f\chi_{\rn\backslash
		(2B)})\|_{L_{p}(B,v)}.
	\end{equation}
	By Lemma \ref{lem geydime}, we get
	\begin{align}\label{eq9085620986}
	\|M_{\a}(f\chi_{(2B)})\|_{L_{p}(B,v)} & \lesssim
	v(B)^{1/p}|B|^{\a/n-1}\|f\|_{L_1(2B)} \notag \\
	&\thickapprox v(B)^{1/p}\|f\|_{L_1(2B)} \sup_{t>2r}
	|B(x,t)|^{\a/n-1} \notag \\
	&\lesssim v(B)^{{1}/{p}} \sup_{t>2r}
	|B(x,t)|^{\a/n-1}\|f\|_{L_1(B(x,t))},
	\end{align}
	Let $y$ be an arbitrary point from $B.$ If
	$B(y,t)\cap\{\rn\backslash (2B)\}\neq\emptyset,$ then $t>r.$ Indeed,
	if $z\in B(y,t)\cap\{\rn\backslash (2B)\},$ then $t> |y-z|\geq
	|x-z|-|x-y|>2r-r=r.$
	
	On the other hand, $B(y,t)\cap\{\rn\backslash (2B)\}\subset
	B(x,2t).$ Indeed, $z\in B(y,t)\cap\{\rn\backslash (2B)\},$ then we
	get $|x-z|\leq |y-z|+|x-y|< t+r< 2t.$
	
	Hence, for all $y \in B$ we have
	\begin{equation*}
	\begin{split}
	M_{\a} (f\chi_{\rn\backslash (2B)})(y)&=
	\sup_{t>0}\frac{1}{|B(y,t)|^{1-\a/n}}
	\int_{B(y,t)\cap\{\rn\backslash (2B)\}}f(y)dy \\
	& \lesssim \, \sup_{t>
		r}\frac{1}{|B(x,2t)|^{1-\a/n}}\int_{B(x,2t)}f(y)dy\\
	&\thickapprox \, \sup_{t> 2r}
	\frac{1}{|B(x,t)|^{1-\a/n}}\int_{B(x,t)}f(y)dy.
	\end{split}
	\end{equation*}
	Thus
	\begin{equation}\label{eq34698286y}
	\|M_{\a}(f\chi_{\rn\backslash (2B)})\|_{L_{p}(B,v)} \lesssim
	v(B)^{{1} / {p}}\left(\sup_{t>r}|B(x,t)|^{\a/n-1}\int_{B(x,t)}f(y)dy\right).
	\end{equation}
	Combining \eqref{eq9085620986} and \eqref{eq34698286y}, in view of
	\eqref{eq678753456}, we arrive at \eqref{25089672897}.
\end{proof}

The following weighted estimate for fractional maximal functions over balls is true.
\begin{lem}\label{lem4.355749430}
	Let $0<p\leq \infty$, $0\leq \a <n$, $v$ be a weight function on $\rn$
	and $f\in \Llocp$. Then for any ball $B=B(x,r)$ in $\rn$
	\begin{equation}\label{eq1000001}
	\begin{split}
	\|M_{\a}f\|_{L_{p}(B,v)} \gtrsim
	v(B)^{{1} / {p}}\left(\sup_{t>r}|B(x,t)|^{\a/n-1}\int_{B(x,t)}
	f(y)dy\right),
	\end{split}
	\end{equation}
	where constant does not depend on $B$ and $f$.
\end{lem}

\begin{proof}
	If $y \in B(x,r)$ and $t>2r$, then $B(x,{t} / {2})\subset B(y,t)$
	and
	\begin{equation*}\begin{split}
	M_{\a}f(y) &\ge 2^{\alpha-n} \sup_{t>2r}
	\frac{1}{|B(x,{t} / {2})|^{1-\a/n}}\int_{B(x,{t} / {2})}
	f(z) dz \\
	&= 2^{\alpha-n} \sup_{t>r}|B(x,t)|^{\a/n-1}\int_{B(x,t)} f(y)dy.
	\end{split}
	\end{equation*}
	Thus \eqref{eq1000001} holds with constant independent of $B$ and
	$f$.
\end{proof}

The following statement follows from Lemmas \ref{lem.RD} and  \ref{lem4.355749430}.
\begin{thm}\label{thm4.7}
	Let $1 \leq q<\infty$, $0<p<q<\infty$, $0\leq \a <n$, $v\in RD_{q(1-\a/n)}$
	and $f\in \Llocp$. Then for any ball $B=B(x,r)\subset\rn$
	\begin{equation*}
	\|M_{\a}f\|_{L_{p}(B,v)} \thickapprox
	v(B)^{{1} / {p}}\left(\sup_{t>r}|B(x,t)|^{\a/n-1}\int_{B(x,t)}f(y)dy\right),
	\end{equation*}
	where constants in equivalency do not depend on $B$ and $f$.
\end{thm}

%%%%%%%%%%%%%%%%%%%%%%%%%%%%%%%%%%%%%%%%%%%%%%%%%%%%%%%%%%%%%%%%%%%%%%%%%%%%%%%%%%%%%%%%%%%%%%%%%%%%

%%%%%%%%%%%%%%%%%%%%%%%%%%%%%%%%%%%%%%%%%%%%%%%%%%%%%%%%%%%%%%%%%%%%%%%%%%%%%%%%%%%%%%%%%%%%%%%%%%%%

\

\section{Relation between $\|I_{\a}f\|_{L_{p}(Q,v)}$ and $\|M_{\a}f\|_{L_{p}(Q,v)}$}\label{Sect3}

\

In this section, we study the relation between weighted Lebesgue norms of Riesz potential and fractional maximal function over
cubes. For general fractional type operators weighted local estimates were investigated in \cite{torch_2014}.

The following statement in unweighted case was proved in \cite[Lemma 2.2]{gm}.
\begin{lem}\label{lem32745627}
	Let $1<p<\infty$, $v\in A_{\infty}$ and $0<\a<n$. For $f\geq 0$ such that
	$I_{\a}f$ is locally integrable the inequality
	\begin{equation*}
	\|I_{\a}f\|_{L_{p}(Q,v)} \leq C \, \left (v(Q)^{{1} / {p}}(I_{\a}f)_Q + \|M_{\a}f\|_{L_{p}(Q,v)}\right)
	\end{equation*}
	holds with constant $C > 0$ independent of $f$ and cube $Q$.
\end{lem}
\begin{proof}
	By \cite[Proposition 3.3]{Adams}, $(I_{\a}f)^{\#}(x)\leq C \, M_{\a}f(x)$
	holds with constant $C > 0$ independent of $f$ and $x\in \rn$. Using
	this fact, the statement follows from Lemma \ref{lem876346754}.
\end{proof}

Our main two side estimate is formulated in the following theorem.
\begin{thm}\label{thm4.1} Let $1<p<\infty$, $0<\a<n$, $v\in A_{\infty}$  and $f\in \Llocp$. Then for any cube $Q=Q(x_0,r_0)$
	\begin{equation}\label{eq041}
	\|I_{\a}f\|_{L_{p}(Q,v)} \thickapprox \|M_{\a}f\|_{L_{p}(Q,v)} + v(Q)^{{1} / {p}}\int_{\rn\backslash Q}\frac{f(y)dy}{|y-x_0|^{n-\a}},
	\end{equation}
	where constants in equivalency do not depend on $Q$ and $f$.
\end{thm}
\begin{proof}
	Let $Q=Q(x_0,r_0)$ be any cube in $\rn$. It's clear that $x\in Q/2$,
	$y\in \rn\backslash Q$ implies $|y-x|\approx |y-x_0|$. Therefore
	\begin{equation}\label{eq16}
	\|I_{\a}(f\chi_{\dual{Q}})\|_{L_{p}(Q/2,v)} \thickapprox v(Q/2)^{{1} / {p}}\int_{\rn\backslash	Q}\frac{f(y)}{|y-x_0|^{n-\a}}dy.
	\end{equation}
	
	Since $v\in \mathcal D$, by
	\eqref{eq16}
	\begin{equation*}
	\|I_{\a}f\|_{L_{p}(Q,v)}\ge	\|I_{\a}(f\chi_{\dual{Q}})\|_{L_{p}(Q/2,v)} \thickapprox v(Q)^{{1} / {p}}\int_{\rn\backslash Q}\frac{f(y)}{|y-x_0|^{n-\a}}dy,
	\end{equation*}
	and by \eqref{eq001}
	$$
	\|I_{\a}f\|_{L_{p}(Q,v)}\gtrsim \|M_{\a}f\|_{L_{p}(Q,v)},
	$$
	we get
	$$
	\|I_{\a}f\|_{L_{p}(Q,v)} \gtrsim \|M_{\a}f\|_{L_{p}(Q,v)} +	v(Q)^{{1} / {p}}\int_{\rn\backslash Q}\frac{f(y)dy}{|y-x_0|^{n-\a}}.
	$$
	
	Let us now to prove that
	$$
	\|I_{\a}f\|_{L_{p}(Q,v)}\lesssim\|M_{\a}f\|_{L_{p}(Q,v)} +	v(Q)^{{1} / {p}}\int_{\rn\backslash	Q}\frac{f(y)dy}{|y-x_0|^{n-\a}}.
	$$
	
	At first note that if
	$$
	\int_{\rn\backslash (2Q)}\frac{|f(y)|}{|y-x_0|^{n-\a}}dy=\infty,
	$$
	then by \eqref{eq16}
	$$
	\|I_{\a}f\|_{L_{p}(Q,v)}\ge\|I_{\a}(f\chi_{\dual{(2Q)}})\|_{L_{p}(Q,v)}	= \infty,
	$$
	then there is nothing to prove.
	
	Assume now that
	$$
	\int_{\rn\backslash (2Q)} \frac{f(y)}{|y-x_0|^{n-\a}} \,dy < \infty.
	$$
	In this case from \cite[Lemma 2.5]{gm} (applied with $p_1=p_2=1$) it follows that $I_{\a}f$ is integrable over $Q$ and the inequality
	\begin{equation}\label{eq.main_1}
	\|I_{\alpha}f\|_{L_1(Q)} \le c \, |Q| \int_{\rn\backslash (2Q)}\frac{f(y)dy}{|y-x_0|^{n-\a}}dy + c|Q|^{{\a}/{n}}\|f\|_{L_1(2Q)}
	\end{equation}
	holds.
	
	By Lemma \ref{lem32745627}, on using \eqref{eq.main_1}, we get that
	\begin{align*}
	\|I_{\a}f\|_{L_{p}(Q,v)}&\leq c \, \|M_{\a}f\|_{L_{p}(Q,v)} + c\, v(Q)^{{1} / {p}}|Q|^{{\a}/{n}-1}\|f\|_{L_1(2Q)} + cv(Q)^{{1} / {p}}\left(\int_{\rn\backslash
		(2Q)}\frac{f(y)dy}{|y-x_0|^{n-\a}}dy\right)\\
	& \leq c \, \|M_{\a}f\|_{L_{p}(Q,v)} + c \,v(Q)^{{1} / {p}} \inf_{x\in Q} M_{\a}f(x) +cv(Q)^{{1} / {p}}\left(\int_{\rn\backslash (2Q)}\frac{f(y)dy}{|y-x_0|^{n-\a}}dy\right)	\\
	& \leq c \, \|M_{\a}f\|_{L_{p}(Q,v)} + c \, v(Q)^{{1} / {p}}\left(\int_{\rn\backslash (2Q)}\frac{f(y)dy}{|y-x_0|^{n-\a}}dy\right).
	\end{align*}
\end{proof}

\begin{rem}\label{rem3475y37}
	Since for any function $f\geq 0$ with compact support in $\rn$ such
	that $I_{\a}f\in \Lloc$
	$$
	(I_{\a}f)^{\#}(x)\thickapprox M_{\a}f(x),~~ \mbox{for any}~~ x\in\rn
	$$
	(see \cite[Proposition 3.3 and 3.4]{Adams} or \cite[Lemma 4.1
	(i)]{AdXi}), then the inequality \eqref{eq041} could be written in
	the following form
	\begin{equation}\label{eq.1.1}
	\|I_{\a}f\|_{L_{p}(Q,v)} \thickapprox
	\|(I_{\a}f)^{\#}\|_{L_{p}(Q,v)} +
	v(Q)^{{1} / {p}}\int_{\rn\backslash
		Q}\frac{f(y)dy}{|y-x_0|^{n-\a}}.
	\end{equation}
	Inequality \eqref{eq.1.1} was proved in \cite[Lemma 1.7]{gm}, when $v \equiv 1$.
\end{rem}

\begin{lem}\label{lem3.4}
	Let $1<p<\infty$, $0<\a<n$, $v\in A_{\infty}$ and $f\in \Llocp$. Then for
	any cube $Q\subset\rn$
	\begin{equation*}
	\|I_{\a}(f\chi_{(2Q)})\|_{L_{p}(Q,v)} \thickapprox	\|M_{\a}(f\chi_{(2Q)})\|_{L_{p}(Q,v)},
	\end{equation*}
	where constants in equivalency do not depend on $Q$ and $f$.
\end{lem}

\begin{proof}
	Let $Q=Q(x_0,r_0)$. In view of \eqref{eq001} we need to show that
	\begin{equation*}
	\|I_{\a}(f\chi_{(2Q)})\|_{L_{p}(Q,v)} \lesssim \|M_{\a}(f\chi_{(2Q)})\|_{L_{p}(Q,v)}.
	\end{equation*}
	By Theorem \ref{thm4.1}, we have
	\begin{equation*}
	\|I_{\a}(f\chi_{2Q})\|_{L_{p}(Q,v)} \lesssim \|M_{\a}(f\chi_{2Q})\|_{L_{p}(Q,v)} +
	v(Q)^{{1} / {p}}\int_{(2Q) \backslash Q}\frac{f(y)dy}{|y-x_0|^{n-\a}}
	\end{equation*}
	But if $y\in (2Q) \backslash Q$, then $|y-x_0| \approx r_0$. Hence
	\begin{equation*}
	\begin{split}
	v(Q)^{{1} / {p}}\int_{(2Q) \backslash	Q}\frac{f(y)dy}{|y-x_0|^{n-\a}} & \approx v(Q)^{{1} / {p}}\frac{1}{|Q|^{1-{\a} / {n}}}\int_{2Q}f(y)dy	\\
	& \lesssim v(Q)^{{1} / {p}} \inf_{x\in	Q} M_{\a}(f\chi_{(2Q)})(x) \lesssim \|M_{\a}(f\chi_{(2Q)})\|_{L_{p}(Q,v)}.
	\end{split}
	\end{equation*}
\end{proof}

\begin{rem}\label{rem4.4}
	It is easy to see that all statements in this section hold true for
	balls instead of cubes.
\end{rem}

The following weighted estimate for fractional maximal functions over balls is true.
\begin{lem}\label{lem4.35574943}
	Let $0<p\leq \infty$, $0\leq \a <n$, $v$ be a weight function on $\rn$
	and $f\in \Llocp$. Then for any ball $B=B(x,r)$ in $\rn$
	\begin{equation}\label{eq100001}
	\|M_{\a}f\|_{L_{p}(B,v)} \gtrsim v(B)^{{1} / {p}} \left(\sup_{t>r}|B(x,t)|^{\a/n-1}\int_{B(x,t)} f(y) \, dy\right),
	\end{equation}
	where constant does not depend on $B$ and $f$.
\end{lem}
\begin{proof}
	If $y \in B(x,r)$ and $t>2r$, then $B(x,{t} / {2})\subset B(y,t)$
	and
	\begin{align*}
	M_{\a}f(y) & \ge 2^{\alpha-n} \sup_{t>2r} \frac{1}{|B(x,{t} / {2})|^{1-\a/n}}\int_{B(x,{t} / {2})} f(z) \, dz \\
	& = 2^{\alpha-n} \sup_{t>r}|B(x,t)|^{\a/n-1}\int_{B(x,t)} f(y)dy.
	\end{align*}
	Thus \eqref{eq100001} holds with constant independent of $B$ and $f$.
\end{proof}

%%%%%%%%%%%%%%%%%%%%%%%%%%%%%%%%%%%%%%%%%%%%%%%%%%%%%%%%%%%%%%%%%%%%%%%%%%%%%%%%%%%%%%%%%%%%%%%%%%%%%%%%%%%%%%%%%%%%%%%%%%%%%%%%%%%%%%%%%
%%%%%%%%%%%%%%%%%%%%%%%%%%%%%%%%%%%%%%%%%%%%%%%%%%%%%%%%%%%%%%%%%%%%%%%%%%%%%%%%%%%%%%%%%%%%%%%%%%%%%%%

\

\section{Characterization of the two-operator weighted norm inequality}\label{Sect5}

\

In this section, we present characterization of the inequality
\eqref{eq3248579857610}. The following theorem holds.
\begin{thm}\label{lem5.5}
	Let $u_0$, $u_1$ and $u_2$ be a continuous weight function defined on
	$(0,\infty)$. Assume that $v_1$ and $v_2$ are weight functions on $\rn$. Then the inequality 
	\begin{equation}\label{eq324857985761}
	\sup_{ r>0 }u_2(r)\int_{\rn\backslash	\Br} g(y) v_2(y) \, dy \le C \sup_{r>0} u_1(r) \left(\sup_{t>r} u_0(t)\int_{\Bt} g(y)v_1(y)\, dy\right).
	\end{equation}
	holds for any non-negative measurable functions $g$ on $\rn$ if and only if 
	\begin{align*}
	I : = \sup_{ r>0 } U_2(r) \bigg(   \big(U_1(r)\big)^{-1} \bigg( \esup_{y\in \rn\backslash \Br} \bigg( v_1(y)^{-1} v_2(y) \bigg) & \notag \\
	& \hspace{-3cm} + \int_r^{\infty}\esup_{y\in \rn\backslash \Bt} \bigg( v_1(y)^{-1} v_2(y) \bigg) \,d \,\big( U_1(t)\big)^{-1}  \bigg) < \infty,
	\end{align*}
	where
	$$
	U_1 (r) := \sup_{r < t < \infty } u_0(t) \bigg( \sup_{0<s<t} u_1(s) \bigg), \qquad U_2 (r) := \sup_{0<t<r} u_2(t), \qquad r > 0.
	$$
	
	Moreover, the best constant in \eqref{eq324857985761}, that is,  
	$$
	B : = \sup_{g \ge 0} \frac{	\sup_{ r>0 }u_2(r)\int_{\rn\backslash	\Br} g(y) v_2(y) \, dy}{\sup_{r>0} u_1(r) \left(\sup_{t>r} u_0(t)\int_{\Bt} g(y)v_1(y)\, dy\right)}
	$$
	satisfies $B \approx I$.
\end{thm}

\begin{proof} 
	Recall that, if $F$ is a non-negative non-increasing function on $\I$, then
	\begin{equation}\label{eq5.2}
	\esup_{t \in (0,\infty)} F(t)G(t) = \esup_{t \in (0,\infty)} F(t)
	\esup_{\tau \in (0,t)} G(\tau);	
	\end{equation}
	likewise, when $F$ is a non-negative non-decreasing function on $\I$, then
	\begin{equation}\label{eq5.3}
	\esup_{t \in (0,\infty)} F(t)G(t) = \esup_{t \in (0,\infty)} F(t)
	\esup_{\tau \in (t,\infty)} G(\tau)
	\end{equation}
	(see, for instance, \cite[p. 85]{gp2}).

	Applying \eqref{eq5.2} and \eqref{eq5.3}, we obtain that
	\begin{align*}
	\sup_{ r>0 }u_2(r)\int_{\rn\backslash	\Br} g(y) v_2(y) \, dy & = \sup_{ r>0 } \left( \sup_{0 < t < r} u_2(t) \right) \int_{\rn\backslash	\Br} g(y) v_2(y) \, dy\\
	\intertext{and}
	\sup_{r>0} u_1(r) \left(\sup_{t>r} u_0(t)\int_{\Bt} g(y)v_1(y)\, dy\right) & = \sup_{r>0} \left( \sup_{0<s<r} u_1(s) \right) \left(\sup_{t>r} u_0(t)\int_{\Bt} g(y)v_1(y)\, dy\right) \\
	& = \sup_{r>0} \left( \sup_{0<s<r} u_1(s) \right) u_0(r) \int_{\Br} g(y)v_1(y)\, dy \\
	& = \sup_{r>0} \left( \sup_{r < t < \infty }\left( \sup_{0<s<t} u_1(s) \right) u_0(t) \right) \int_{\Br} g(y)v_1(y)\, dy.
	\end{align*}
	
	We can rewrite inequality \eqref{eq324857985761} in the following
	form:
	\begin{equation}\label{eqd;lbnslfgnb xcmvn}
	\sup_{ r>0 }U_2(r)\int_{\rn\backslash	\Br} g(y) v_2(y) \, dy \le C \sup_{r>0} U_1(r) \int_{\Br} g(y)v_1(y)\, dy.
	\end{equation}
	
	Obviously, inequality \eqref{eqd;lbnslfgnb xcmvn} is equivalent to the inequality
	\begin{equation}\label{eq.equiv.}
	\sup_{ r>0 }U_2(r)\int_{\rn\backslash	\Br} g(y) v_1(y)^{-1} v_2(y)\, dy \lesssim \sup_{r>0} U_1(r) \int_{\Br} g(y) \, dy,
	\end{equation}
	and
	\begin{align*}
	B = \sup_{ g\geq 0}\frac{\sup_{ r>0 }U_2(r)\int_{\rn\backslash	\Br} g(y) v_1(y)^{-1} v_2(y)\, dy}{\sup_{r>0} U_1(r) \int_{\Br} g(y)\, dy}.
	\end{align*}
	
	Observe that if $U_1 (t)=\infty$ for some $t \in (0,\infty)$,
	then \eqref{eq.equiv.} holds trivially and the inequality
	\eqref{eq324857985761} also satisfied. Assume that $U_1 (t)<\infty$ for
	all $t \in (0,\infty)$.
	
	Interchanging the suprema, we have that
	\begin{align*}
	B & = \sup_{ r>0 } U_2(r)\sup_{ g\geq 0}\frac{\int_{\rn} g(y) \chi_{\rn\backslash	\Br} (y) v_1(y)^{-1} v_2(y)\, dy}{\sup_{r>0} U_1(r) \int_{\Br} g(y)\, dy}.
	\end{align*}
	
	By \cite[Theorem 2.4]{gm1},	we get that
	\begin{align*}
	B = & \sup_{ r>0 } U_2(r) \left( \int_0^{\infty}\esup_{y\in \rn\backslash \Bt} \left(\chi_{\rn\backslash	\Br} (y) v_1(y)^{-1} v_2(y)\right)d\big(U_1(t)\big)^{-1} \right. \\
	& \left. + \big(U_1(0)\big)^{-1}\esup_{y\in \rn}\left(\chi_{\rn\backslash	\Br} (y) v_1(y)^{-1} v_2(y)\right) \right). 
	\end{align*}
	
	Hence
	\begin{align*}
	B = & \sup_{ r>0 } U_2(r) \bigg( \bigg( \esup_{y\in \rn\backslash \Br} v_1(y)^{-1} v_2(y) \bigg) \,\int_0^r d \big( U_1(t) \big)^{-1}  \\
	& + \int_r^{\infty}\esup_{y\in \rn\backslash \Bt} \bigg( \chi_{\rn\backslash	\Br} (y) v_1(y)^{-1} v_2(y) \bigg) \,d \,\big( U_1(t)\big)^{-1} \\
	& + (U_1(0))^{-1} \bigg( \esup_{y\in \rn\backslash \Br} v_1(y)^{-1} v_2(y) \bigg) \\
	= & \sup_{ r>0 } U_2(r) \bigg(   \big(U_1(r)\big)^{-1} \bigg( \esup_{y\in \rn\backslash \Br} \bigg( v_1(y)^{-1} v_2(y) \bigg) \\
	& + \int_r^{\infty}\esup_{y\in \rn\backslash \Bt} \bigg( v_1(y)^{-1} v_2(y) \bigg) \,d \,\big( U_1(t)\big)^{-1}  \bigg).
	\end{align*}
\end{proof}

\begin{cor}\label{cor5.5}
	Let $0 < \beta < \infty$ and $u$ be a continuous weight function defined on
	$(0,\infty)$. Then the inequality 
	\begin{equation}\label{eq.cor}
	\sup_{ r>0 }u(r)\int_{\rn\backslash	\Br} |y|^{-\b}g(y) \,dy \lesssim \sup_{r>0}u(r) \left(\sup_{t>r}t^{-\b}\int_{\Bt} g(y) \,dy\right)
	\end{equation}
	holds for any non-negative measurable functions $g$ on $\rn$ if and only if 
	\begin{align*}
	\sup_{r>0} r^{\beta} \,\bigg( \sup_{r < t < \infty } t^{-\beta} \bigg( \sup_{0<s<t} u(s) \bigg)\bigg) \,\int_r^{\infty} t^{-\beta - 1} \,\bigg( \sup_{t < \tau < \infty } \tau^{-\beta} \bigg( \sup_{0<s<\tau} u(s) \bigg)\bigg)^{-1}  \,dt < \infty.
	\end{align*}
\end{cor}

\begin{proof} 
	Applying Theorem \ref{lem5.5} with $v_1(y) \equiv 1 $, $v_2 (y) = |y|^{-\beta}$, $u_0(t) = t^{-\beta}$, $u_1 (t) = u_2(t) = u(t)$, taking into account that
	$$
	U_1 (r) = \sup_{r < t < \infty } t^{-\beta} \bigg( \sup_{0<s<t} u(s) \bigg), \qquad U_2 (r) = \sup_{0<t<r} u(t), \qquad r > 0,
	$$
	we have
	\begin{align*}
	I = \sup_{ r>0 } U_2(r) \bigg(   \big(U_1(r)\big)^{-1} r^{-\beta} + \int_r^{\infty}t^{-\beta} \,d \,\big( U_1(t)\big)^{-1}  \bigg).
	\end{align*}
	
	Since
	$$
	\sup_{ r>0 } U_2(r) \big(U_1(r)\big)^{-1} r^{-\beta} \le \sup_{ r>0 } U_1(r) \big(U_1(r)\big)^{-1} \le 1,
	$$
	inequality \eqref{eq.cor} holds if and only if
	$$
	\sup_{ r>0 } U_2(r) \bigg( \int_r^{\infty}t^{-\beta} \,d \,\big( U_1(t)\big)^{-1}  \bigg) < \infty.
	$$
	
	But the latter holds if and only if 
	$$
	\int_r^{\infty} t^{-\beta - 1} \,\big( U_1(t)\big)^{-1}  \,dt  \lesssim r^{-\beta} \,\big( U_1(r)\big)^{-1}, \quad r > 0
	$$
	(cf.  the proof of \cite[Theorem 4.1]{gm}).
	
	The proof is completed.
\end{proof}

\begin{cor}\label{lem5.5.0}
	Let $0<\beta <n$  and $u$ be a continuous weight function defined on
	$\rn\times(0,\infty)$. 
	If
	\begin{equation} \label{4.10}
	\sup_{x\in \rn}\sup_{0<s<\infty}s^{\b}\left(\sup_{s<y<\infty}s^{-\beta}\sup_{0<\tau<y}u(x,\tau)\right)\int_s^{\infty}t^{-\b-1}
	\left(\sup_{t<z<\infty}z^{-\beta}\sup_{0<\tau<z}u(x,\tau)\right)^{-1}dt<\infty,
	\end{equation}
	then the inequality
	\begin{align}\label{eq32485798576100}
	\sup_{x\in\rn}\sup_{ r>0 }u(x,r)\int_{\rn\backslash
		\Bxr}\frac{g(y)}{|x-y|^{\b}}dy
	% &\\\&\hspace{-1cm}
	\lesssim \sup_{x\in\rn}\sup_{r>0}u(x,r)
	\left(\sup_{t>r}t^{-\b}\int_{\Bxt} g(y)dy\right).
	\end{align}
	holds for any non-negative measurable functions $g$ on $\rn$.
\end{cor}

\begin{proof} 
	Assume that \eqref{4.10} holds. Fix any $x \in \rn$ and $g \in {\mathfrak M}^+(\rn)$. Applying Corollary \ref{cor5.5} to the function $g(x + \cdot)$, we get that the inequality
	\begin{equation*}
	\sup_{ r>0 }u(r)\int_{\rn\backslash	\Br} \frac{g(x+y)}{|y|^{\b}} \,dy \lesssim \sup_{r>0}u(r) \left(\sup_{t>r}t^{-\b}\int_{\Bt} g(x+y) \,dy\right)
	\end{equation*}
	holds with the constant independent of $x$. Hence the inequality	
	\begin{align*}
    \sup_{ r>0 }u(x,r)\int_{\rn\backslash \Bxr}\frac{g(y)}{|x-y|^{\b}}dy \lesssim \sup_{r>0}u(x,r)
	\left(\sup_{t>r}t^{-\b}\int_{\Bxt} g(y)dy\right).
	\end{align*}
	holds as well. Consequently, inequality \eqref{eq32485798576100} holds.
\end{proof}

%%%%%%%%%%%%%%%%%%%%%%%%%%%%%%%%%%%%%%%%%%%%%%%%%%%%%%%%%%%%%%%%%%%%%%%%%%%%%%%%%%%%%%%%%%%%%%%%%%%%%%%%%%%%%%%%%%%%%%%%%%%%%%%%%%%%%%%%%
%%%%%%%%%%%%%%%%%%%%%%%%%%%%%%%%%%%%%%%%%%%%%%%%%%%%%%%%%%%%%%%%%%%%%%%%%%%%%%%%%%%%%%%%%%%%%%%%%%%%%%%%%%%%%%%%%%%%%%%%%%%%%%%%%%%%%%

\

\section{Main results}\label{Sect6}

\

In this section, we extend the theorem of Muchenhoupt-Wheeden to generalized weighted Morrey spaces ${\mathcal M}_{p,\o}(\rn,v)$ and generalized weighted central Morrey space $\dot{\mathcal
	M}_{p,\o}(\rn,v)$. 

Our first main result reads as follows.
\begin{thm}\label{equiv}
	Let $1<p<\infty$, $0<\a<n$, $v\in A_{\infty}$ and $\o$ be a continuous
	weight function defined on $\rn\times(0,\infty)$.
	If
	\begin{equation*}
	\sup_{x \in \rn} \sup_{s > 0} s^{n - \a}\psi(x,s)\int_s^{\infty}t^{\a-n-1} \left(\psi(x,t)\right)^{-1}dt < \infty,
	\end{equation*}
	where
	$$
	\psi(x,t) : =\sup_{t<s<\infty}s^{\a-n}\sup_{0<\tau<s} \o(x,\tau)\, v(B(x,\tau))^{1/p},
	$$
	then there exists a constant $C$ such that for any 	$f\in\Llocp$
	\begin{equation*}
	C^{-1} \|M_{\a}f\|_{{\mathcal M}_{p,\o }(\rn,v)} \le \|I_{\a}f\|_{{\mathcal M}_{p,\o }(\rn,v)} \le C \|M_{\a}f\|_{{\mathcal M}_{p,\o }(\rn,v)}.
	\end{equation*}
\end{thm}

\begin{proof}
	In view of
	\eqref{eq001} we need only to prove
	\begin{equation*}
	\|I_{\a}f\|_{{\mathcal M}_{p,\o }(\rn,v)} \lesssim \|M_{\a}f\|_{{\mathcal M}_{p,\o }(\rn,v)}.
	\end{equation*}
	With regard to Theorem \ref{thm4.1} (see Remark \ref{rem4.4}) and
	Lemma \ref{lem4.35574943}, it will suffice to show that
	\begin{align*}
	\sup_{x\in \rn, \; r>0 }\o(x,r) v(B(x,r))^{{1} / {p}}\int_{\rn\backslash	B(x,r)}\frac{f(y)}{|y-x|^{n-\a}}dy & \\
	& \hspace{-5cm}\lesssim\sup_{x\in \rn, \; r>0}\o(x,r) v(B(x,r))^{{1} / {p}}\left(\sup_{t>r}\frac{1}{|B(x,t)|^{1-{\a} / {n}}}\int_{B(x,t)}f(y)dy\right). 
	\end{align*}
	The latter immediately follows from Corollary
	\ref{lem5.5.0}, applied with $u(x,r) = \o(x,r) v(B(x,r))^{ 1 / p}$ and $\beta = n - \alpha$.
\end{proof}

As a special case (when $v=1$), we obtain Theorem \ref{thmGM}.

%%%%%%%%%%%%%%%%%%%%%%%%%%%%%%%%%%%%%%%%%%%%%%%%%%%%%%%%%%%%%%%%%%%%%%%%%%%%%%%%%%%%%%%%%%%%

Now we restrict our attention to the generalized weighted central Morrey spaces. In this case we are able to get criteria for the equivalency of norms of the Riesz potential and the fractional
maximal function.
\begin{thm}\label{thm.main.1}
	Let $1<p<\infty$, $0<\a<n$, $v\in A_{\infty}$, $\o$ be a continuous weight
	function defined on $(0,\infty)$. Then for any $f\in\Llocp$
	\begin{equation*}
	\|I_{\a}f\|_{\dot{\mathcal M}_{p,\o }(\rn,v)} \thickapprox
	\|M_{\a}f\|_{\dot{\mathcal M}_{p,\o }(\rn,v)}
	\end{equation*}
    if and only if
	\begin{align}\label{eq.two-op.1}
	\sup_{r>0 }	\omega (r) w(B(0,r))^{{1} / {p}} \int_{\rn\backslash \Br} |y|^{\a - n} f(y)\,dy \lesssim \sup_{r>0} \o(r) \|M_{\a}f\|_{L_{p,w}(\Br)}.
	\end{align}
\end{thm}

\begin{proof}
	By inequality \eqref{eq001} and Theorem \ref{thm4.1}, we have that
	$$
	\|M_{\a}f\|_{L_{p}(\Br,v)} \lesssim \|I_{\a}f\|_{L_{p}(\Br,v)} \thickapprox \|M_{\a}f\|_{L_{p}(\Br,v)} + v(\Br)^{{1} / {p}} \int_{\rn\backslash \Br} |y|^{\a - n} f(y)\,dy, \quad r > 0.
	$$
	Thus, if inequality \eqref{eq.two-op.1} holds, then inequalities
	\begin{align*}
	\|M_{\a}f\|_{\dot{\mathcal M}_{p,\o }(\rn,v)} = \sup_{r > 0} \omega (r) \|M_{\a}f\|_{L_{p}(\Br,v)} \lesssim \sup_{r > 0} \omega (r) \|I_{\a}f\|_{L_{p}(\Br,v)} = \|I_{\a}f\|_{\dot{\mathcal M}_{p,\o }(\rn,v)},
	\end{align*}
	and 
	\begin{align*}
	\|I_{\a}f\|_{\dot{\mathcal M}_{p,\o }(\rn,v)} & = \sup_{r > 0} \omega (r) \|I_{\a}f\|_{L_{p}(\Br,v)} \\
	& \lesssim \sup_{r > 0} \omega (r) \|M_{\a}f\|_{L_{p}(\Br,v)} + \sup_{r > 0} \omega (r) v(\Br)^{{1} / {p}} \int_{\rn\backslash \Br} |y|^{\a - n} f(y)\,dy \\
	& \lesssim \sup_{r > 0} \omega (r) \|M_{\a}f\|_{L_{p}(\Br,v)} = \|M_{\a}f\|_{\dot{\mathcal M}_{p,\o }(\rn,v)}
	\end{align*}
	hold.
	
	Now assume that
	$$
	\|I_{\a}f\|_{\dot{\mathcal M}_{p,\o }(\rn,v)} \thickapprox
	\|M_{\a}f\|_{\dot{\mathcal M}_{p,\o }(\rn,v)}
	$$
	holds. By Theorem \ref{thm4.1}, we know that
	$$
	v(B(0,r))^{{1} / {p}} \int_{\rn\backslash \Br} |y|^{\a - n} f(y)\,dy \lesssim \|I_{\a}f\|_{L_{p}(\Br,v)}.
	$$
	Hence, we have that
	$$
	\sup_{r>0 }	\omega (r) v(B(0,r))^{{1} / {p}} \int_{\rn\backslash \Br} |y|^{\a - n} f(y)\,dy \lesssim \sup_{r>0} \o(r) \|I_{\a}f\|_{L_{p}(\Br,v)} = \|I_{\a}f\|_{\dot{\mathcal M}_{p,\o }(\rn,v)}.
	$$
	Consequently, the inequality
	$$
	\sup_{r>0 }	\omega (r) v(B(0,r))^{{1} / {p}} \int_{\rn\backslash \Br} |y|^{\a - n} f(y)\,dy \lesssim 
	\|M_{\a}f\|_{\dot{\mathcal M}_{p,\o }(\rn,v)}
	$$
	holds as well.
	
	The proof is completed.
\end{proof}

\begin{thm}\label{main_2}
	Let $1 < p < q < \infty$, $0<\a<n$, $v\in A_{\infty} \cap RD_{q(1-\a/n)}$, $\o$ be a continuous weight
	function defined on $(0,\infty)$. Then for any $f\in\Llocp$
	\begin{equation*}
	\|I_{\a}f\|_{\dot{\mathcal M}_{p,\o }(\rn,v)} \thickapprox
	\|M_{\a}f\|_{\dot{\mathcal M}_{p,\o }(\rn,v)}
	\end{equation*}
	if and only if
	\begin{align*}
	\sup_{ r>0 } \omega (r) v(B(0,r))^{{1} / {p}} \int_{\rn\backslash	\Bxr} |y|^{\alpha - n} f(y) \, dy \lesssim \sup_{r>0} \omega (r) v(B(0,r))^{{1} / {p}} \left(\sup_{t>r}t^{\alpha - n}\int_{\Bt} f(y)dy\right).
	\end{align*}
\end{thm}

\begin{proof}
	Assume that the inequality 
	$$
	\sup_{ r>0 } \omega (r) v(B(0,r))^{{1} / {p}} \int_{\rn\backslash	\Bxr} |y|^{\alpha - n} f(y) \, dy \lesssim \sup_{r>0} \omega (r) v(B(0,r))^{{1} / {p}} \left(\sup_{t>r}t^{\alpha - n}\int_{\Bt} f(y)dy\right)
	$$
	holds. Then, by Lemma \ref{lem4.35574943}, the inequality
	$$
	\sup_{ r>0 } \omega (r) v(B(0,r))^{{1} / {p}} \int_{\rn\backslash	\Bxr} |y|^{\alpha - n} f(y) \, dy \lesssim \sup_{r > 0} \omega (r) \|M_{\a}f\|_{L_{p}(\Br,v)} = \|M_{\a}f\|_{\dot{\mathcal M}_{p,\o }(\rn,v)}
	$$ 
	holds as well. Therefore, by Theorem \ref{thm.main.1}, we have that
	$$
	\|I_{\a}f\|_{\dot{\mathcal M}_{p,\o }(\rn,v)} \thickapprox \|M_{\a}f\|_{\dot{\mathcal M}_{p,\o }(\rn,v)}.
	$$
	
	Now assume that
	$$
	\|I_{\a}f\|_{\dot{\mathcal M}_{p,\o }(\rn,v)} \thickapprox \|M_{\a}f\|_{\dot{\mathcal M}_{p,\o }(\rn,v)}
	$$
	holds. Then, by Theorem \ref{thm.main.1}, the inequality
	$$
	\sup_{ r>0 } \omega (r) v(B(0,r))^{{1} / {p}} \int_{\rn\backslash	\Bxr} |y|^{\alpha - n} f(y) \, dy \lesssim \|M_{\a}f\|_{\dot{\mathcal M}_{p,\o }(\rn,v)}
	$$ 
	holds. Hence, by Lemma \ref{lem.RD}, the inequality
	$$
	\sup_{ r>0 } \omega (r) v(B(0,r))^{{1} / {p}} \int_{\rn\backslash	\Bxr} |y|^{\alpha - n} f(y) \, dy \lesssim \sup_{r>0} \omega (r) v(B(0,r))^{{1} / {p}} \left(\sup_{t>r}t^{\alpha - n}\int_{\Bt} f(y)dy\right)
	$$
	holds as well.
	
	The proof is completed.
\end{proof}

Our second main result is formulated as follows.
\begin{thm}\label{equiv_2}
	Let $1 < p < q < \infty$, $0<\a<n$, $v\in A_{\infty} \cap RD_{q(1-\a/n)}$, $\o$ be a continuous weight
	function defined on $(0,\infty)$. Then
	\begin{equation*}
	\|I_{\a}f\|_{\dot{\mathcal M}_{p,\o }(\rn,v)} \thickapprox
	\|M_{\a}f\|_{\dot{\mathcal M}_{p,\o }(\rn,v)}
	\end{equation*}
	if and only if 
	\begin{equation*}
	\sup_{r > 0} r^{n - \alpha} \,\bigg( \sup_{r < t < \infty } t^{\alpha - n} \bigg( \sup_{0<s<t} u(s) \bigg)\bigg) \int_r^{\infty} t^{\alpha - n - 1} \,\bigg( \sup_{t < \tau < \infty } \tau^{\alpha - n} \bigg( \sup_{0<s<\tau} u(s) \bigg)\bigg)^{-1}  \,dt  < \infty.	
	\end{equation*}
\end{thm}

\begin{proof}
The statement follows by Theorem \ref{main_2} and Corollary \ref{cor5.5}, applied with $u(r) = \o(r) v(B(0,r))^{ 1 / p}$ and $\beta = n - \alpha$.
\end{proof}	

%%%%%%%%%%%%%%%%%%%%%%%%%%%%%%%%%%%%%%%%%%%%%%%%%%%%%%%%%%%%%%%%%%%%%%%%%%%%%%%%%%%%%%%%%%%%%%%%%%%%%%%%%%%%%%%%%

%%%%%%%%%%%%%%%%%%%%%%%%%%%%%%%%%%%%%%%%%%%%%%%%%%%%%%%%%%%%%%%%%%%%%%%%%%%%%%%%%%%%%%%%%%%%%%%%%%%%%%%%%%%%%%%%%%%%%%%%%%%%%%%%%%%%%%

%%%%%%%%%%%%%%%%%%%%%%%%%%%%%%%%%%%%%%%%%%%%%%%%%%%%%%%%%%%%%%%%%%%%%%%%%%%%%%%%%%%%%%%%%%%%%%%%%%%%%%%%%%%%%%%%

\end{document}